\documentclass[12pt,reqno,a4paper]{amsart}
\usepackage{color}
\usepackage{amssymb}
\usepackage{eucal}
\usepackage{amsmath}
\usepackage{amsthm}
\usepackage{amscd}
\usepackage{graphics}
\usepackage{graphicx}
\usepackage{amsfonts}
\usepackage{latexsym}
\usepackage[all]{xy}
\usepackage{subfig}
\usepackage{float} % For using [H] placement specifier
\usepackage{graphicx} % For including graphics
\usepackage{hyperref}
\hypersetup{
	colorlinks   =  true,
	linkcolor    = cyan,
	citecolor    = red,
	urlcolor	=magenta,     
}
\newcommand\on{\operatorname}

\DeclareMathOperator{\ad}{ad}

\newcommand\oo{{\infty}}

\renewcommand\o{\circ}

\newcommand\cal{\mathcal}
\newcommand\R{\mathbb{R}}
\renewcommand\d{\mathrm{d}}
\newcommand \al{\alpha}
\newcommand\be{\beta}
\newcommand\ga{\gamma}

\renewcommand\th{\theta}
\newcommand\io{\iota}
\newcommand\ka{\kappa}

\newcommand\ta{\tau}

\newcommand\om{\omega}

\newcommand\Om{\Omega}
\newcommand\M{\mathcal{M}}

\newcommand\pa{\partial}

\newcommand\RR{\mathbb{R}}
\newcommand\CC{\mathbb{C}}

\newcommand\db{\on{DB}}

\newcommand\g{\mathfrak {g}}
\newcommand\X{\mathfrak X}

\renewcommand\Im{\on{Im}}
\newcommand\Ker{\on{Ker}}

\newcommand\ind{\textnormal{ind}}
\newcommand\srp{\sharp_{_\Pi}}

\newcommand\srm{\sharp_{_\M}}
\newcommand\srind{{g}_{\ind}^{S}}

\numberwithin{equation}{section}
\newtheorem{thm}{\bf Theorem}[section]
\newtheorem{lem}[thm]{\bf Lemma}
\newtheorem{prop}[thm]{\bf Proposition}

\newtheorem{defn}{\bf Definition}[section]

\theoremstyle{remark}
\newtheorem{rem}{\bf Remark}[section]

\newtheorem{exmp}{\bf Example}[section]

\definecolor{bgd}{RGB}{153,0,51}      %burgundy



\begin{document}
	
	%\title{ DB vector fields on Poisson manifolds}
	\title{Generalized double bracket vector fields}
	 
 \author[P.\ Birtea]{Petre Birtea}
\address{P.\ Birtea:
Department of Mathematics, West University of Timi\c soara \\
}
\email{petre.birtea@e-uvt.ro}

\author[Z. Ravanpak]{Zohreh Ravanpak}
\address{Z.\ Ravanpak: 
Department of Mathematics, West University of Timi\c soara \\
} 
\email{zohreh.ravanpak@e-uvt.ro}

\author[C. Vizman]{Cornelia Vizman}
\address{C.\ Vizman:
Department of Mathematics, West University of  Timi\c soara \\} \email{cornelia.vizman@e-uvt.ro}

	\begin{abstract}
		\noindent
        We generalize double bracket vector fields, originally defined on semisimple Lie algebras, to Poisson manifolds equipped with a pseudo-Riemannian metric by utilizing a symmetric contravariant 2-tensor field. We extend the normal metric on an adjoint orbit of a compact semisimple Lie algebra to ensure that these vector fields become gradient vector fields on each symplectic leaf. Furthermore, we apply this construction to enhance the equilibria of Hamiltonian systems, specifically addressing the challenge of asymptotically stabilizing points that are already stable, through dissipation terms derived from generalized double bracket vector fields.
        	\end{abstract}
	\maketitle

	%\tableofcontents

	\section{Introduction}
	\medskip\noindent 

    On a Riemannian manifold, a fundamental question is how to compute the gradient vector field of a cost function. Especially when the metric is not simple or when working in non-standard coordinate system. This problem arises in various contexts, particularly in the design of steepest descent algorithms on Riemannian manifolds, such as matrix Lie groups and homogeneous spaces {(e.g.~\cite{helmke})}. To address this question, one requires adequate information about the metric to solve the equation defining the gradient vector field. This process can often be infeasible or lead to highly complex calculations. 
\smallskip\noindent

A more  amenable case arises when the manifold of interest is embedded as a submanifold within an ambient Riemannian manifold possessing simpler geometric properties. In such cases, one can explicitly construct an orthogonal projection operator that relates the gradient vector fields between the two manifolds. The orthogonal projection operator allows for the computation of the gradient on the submanifold by projecting the gradient from the ambient space onto the tangent space of the submanifold. This approach has been widely applied in constrained optimization problems, e.g.~ in \cite{absil, edelman}.

	\smallskip In the often encountered case of a submanifold defined by a set of constraint functions, the gradient vector field on the submanifold, with respect to the induced metric, can be realized as the restriction of a vector field (termed the embedded gradient vector field) defined on the ambient space. This construction has been thoroughly developed in \cite{birtea-comanescu, Birtea-Comanescu-Hessian}.

	  \smallskip
In certain instances, there may be insufficient components of the constraint map that defines the submanifold. This situation arises, for example, when there is inadequate knowledge of the Casimir functions necessary to fully describe a symplectic leaf of a Poisson manifold. Nevertheless, in the context of the Lie-Poisson structure associated with a compact Lie algebra, the gradient vector field on the regular symplectic leaves can be derived as the restriction of a vector field defined in the ambient space, referred to as the double bracket (DB) vector field \cite{bloch-brockett}. In contrast to the previous setting, the metric on the symplectic leaf is not the induced metric; rather, it is known as the normal metric.
  \smallskip
  
Our goal in this paper is to establish an analogue of the aforementioned scenario within the framework of a general Poisson manifold equipped with a (pseudo-)Riemannian metric. More precisely, following an idea presented in \cite{birtea-comanescu}, we construct a symmetric contravariant 2-tensor field that couples the Poisson structure with the Riemannian structure, without imposing compatibility conditions. In the literature, Poisson manifolds endowed with symmetric contravariant 2-tensor fields are referred to as metriplectic structures \cite{mo1,mo2,izu}. These structures provide a geometric framework for dynamics that exhibit both conservative and dissipative characteristics. Similarly, the newly constructed tensor field, which we denote as the metriplectic tensor field, facilitates the association of a vector field to every smooth function in a natural manner. We call this vector field the generalized double bracket (GDB) vector field, as it recovers the classical DB vector field in the case of a compact semisimple Lie algebra. The GDB vector field possesses several advantageous properties: it is tangent to all symplectic leaves and, when restricted to such a leaf, it demonstrates gradient-type behavior with respect to a natural metric that generalizes the normal metric.

 \smallskip  All of the constructions described above can be applied more generally to a pseudo-Riemannian structure on the Poisson manifold. Consequently, non-compact semisimple Lie algebras can also be included in this framework. However, this generalization requires a restriction to what we will refer to as ``good symplectic leaves", specifically those for which the induced metric is non-degenerate.
 
 \smallskip
In the nonlinear Poisson setting, new complications arise, particularly concerning the critical question of whether the restriction of the indefinite signature metric to the symplectic leaves is non-degenerate. While this property is observed in the specific case of a non-compact semisimple Lie algebra on certain exceptional leaves, the situation becomes significantly more intricate in the general case. This question will be addressed in future research, which will explore the sophisticated interplay between the Poisson and pseudo-Riemannian structures.

\smallskip
Our constructions, which generalize the DB vector fields, are well-suited for the investigation of various dynamical properties of dissipative systems, analogous to methodologies employed in \cite{BC,Moris,Bloch}. To demonstrate the practical significance of these newly introduced mathematical objects, we apply them to the problem of asymptotic stabilization. Specifically, we illustrate how these GDB vector fields can be utilized to enhance the stability of already stable equilibria in Hamiltonian systems. This is achieved by incorporating a dissipation term of the GDB vector field type, thereby transforming stable equilibria into asymptotically stable states.

%%%%%%%

	\subsection*{Structure of the paper} 
In Section \ref{sec:2}, we recall the classical DB vector field in the linear case of a semisimple Lie algebra 
$\g$. To generalize this setting, we construct a symmetric tensor field that couples a Poisson structure with a pseudo-Riemannian structure, which we refer to as metriplectic tensor field. Subsequently, we introduce the GDB vector field and demonstrate that, in the specific case of a semisimple Lie algebra, it reduces to the classical DB vector field. In Section \ref{Sec3}, we construct a pseudo-Riemannian metric on symplectic leaves, which we term the DB metric, thus generalizing the normal metric defined on an adjoint orbit of a compact semisimple Lie algebra. It is specifically restricted to what we define as good symplectic leaves within the body of the paper. Furthermore, we demonstrate that the restriction of the GDB vector field to a good symplectic leaf corresponds to the gradient of a smooth function with respect to the DB metric. In Section \ref{sec:4}, we utilize the example of two harmonic oscillators in $(n:m)$ resonance to illustrate how the addition of a GDB vector field to Hamiltonian dynamics transforms stable equilibria into asymptotically stable equilibria, while preserving the structure of the symplectic leaves.

%%%%%%%%%%

	\section{Metriplectic tensor field: Generalized double bracket vector field}\label{sec:2}

     In this section, we introduce the metriplectic tensor field and the GDB vector field. These objects are developed within the general framework of a Poisson manifold $(M, \Pi)$ endowed with a pseudo-Riemannian metric $g$, without imposing any compatibility conditions. As a particular case, we recover the classical DB vector field.
	 
	 \medskip \noindent We begin by recalling the classical setting of the double bracket (DB) vector field. Let $ (\mathfrak{g}, [\ ,\ ]) $ be a semisimple Lie algebra with $\kappa \colon \mathfrak{g} \times \mathfrak{g} \to \mathbb{R}$ representing the Killing form, which is a non-degenerate, symmetric, Ad-invariant bilinear form. The following vector field, referred to as the DB vector field, was introduced by Brockett \cite{brockett-1,brockett-2}, and discussed further in \cite{bloch}, in the context of dynamical numerical algorithms and linear programming:
	\begin{equation}\label{1}
		\dot L=[L,[L,N]]\,,
		\end{equation}
	where $L\in\g$ and $N$ is a fixed regular element in ${\g}$. The DB vector field is indeed tangent to the adjoint orbits of the Lie algebra $\mathfrak{g}$, which are the symplectic leaves for the linear Poisson bracket on $\mathfrak{g}$. More precisely, by identifying the Lie algebra $\mathfrak{g}$ with its dual $\mathfrak{g}^*$ using the Killing form, the Lie-Poisson bracket on $\mathfrak{g}^*$ transforms into the linear Poisson bracket on $\mathfrak{g}$:
	\begin{equation}\label{2}
		\{F,G\}_{{\g}}(L)={\ka}(L,[\nabla F(L),\nabla G(L)])\,. 
	\end{equation}
	
	\medskip 
    In the case of a compact semisimple Lie algebra $\mathfrak{g}$, it has been proved that the DB vector field \eqref{1}, when restricted to a regular adjoint orbit $S \subset \mathfrak{g}$, is a gradient vector field with respect to the normal metric. We recall briefly this construction; for details, see \cite{bloch-brockett,bloch-flaschka}. For every $L \in S$, consider the orthogonal decomposition with respect to the negative Killing form $\kappa$, that is, $\mathfrak{g} = \mathfrak{g}_L \oplus \mathfrak{g}^L$, where $\mathfrak{g}_L = \text{Im}(\ad_L)$ and $\mathfrak{g}^L = \text{Ker}(\ad_L)$. The linear space $\mathfrak{g}_L$ can be identified with the tangent space $T_L S$ and $\mathfrak{g}^L$ with the normal space. One can endow the adjoint orbit $S$ with the normal metric \cite{besse}, also called the standard metric \cite{atiyah},
	\begin{equation}\label{nor}
		{\nu}^{S}([L,X],[L,Y])=-{\ka}(X^L,Y^L)\,,
	\end{equation}
	where $X^L,Y^L$ are the normal components according to the above orthogonal decomposition of $X$ and $Y$, respectively.
	\begin{thm} [\cite{bloch-brockett}, \cite{bloch-flaschka}] \label{bloch-ratiu}
		Let $N$ be a fixed regular element \footnote{The dimension of its centralizer equals the rank of the Lie algebra.} in the compact semisimple Lie algebra $\g$. 
		The gradient of the linear function $H\colon \g\to \mathbb{R}$, defined by $H(L)={\ka}(L,N)$, when restricted to a regular adjoint orbit $S$,  and computed with respect to the normal metric, is given by:	
		$$\nabla_{\nu^S}(H|_S)(L)=[L,[L,N]]\,.$$
	\end{thm}
    
\medskip
 In the following, we consider a more general scenario where $(M,\Pi,{g})$ is a Poisson manifold equipped with a pseudo-Riemannian metric. 
	
	\begin{defn} \label{Mdef}
We define the \textbf{metriplectic tensor field} as the following symmetric contravariant 2-tensor field:
		$$\M(\alpha,\beta):={g}(\sharp_{_\Pi}\al,\srp\beta),\quad\alpha,\beta\in\Om^1(M)\,.$$ 
\end{defn}
\noindent When considering closed forms 
 $\alpha =\d F$ and $\beta = \d G$, where $F,G\in {\cal C}^{\infty}(M)$, the metriplectic tensor field can be expressed as follows:
$${{\M}(\d F,\d G)={g}(X_F,X_G)\,.}$$
\begin{lem}\label{bigdiagram}
	The identity 
	$\srm=-\srp\o\flat_g\o\srp$ holds, 	where $\flat_g$  is the flat operator associated with the pseudo-Riemannian metric $g$ on $M$.
\end{lem}
\begin{proof}
The proof involves the following straightforward computation:
	\begin{align*}
		(\srm\al)(\be)&=\M(\al,\be)=g(\srp\al,\srp\be)=\srp\be(\flat_g(\srp\al))
		=\Pi(\be,\flat_g\o\srp(\al))\\
		&=-\Pi(\flat_g\o\srp(\al),\be)=-(\srp\o\flat_g\o\srp)(\al)(\be)\,,
	\end{align*}
	for all $\al,\be\in \Om^1(M)$\,. 
\end{proof}
\noindent In the finite-dimensional case, the symmetric matrix associated with the metriplectic tensor field is given by:
\begin{equation}\label{prod}
	[{\M}]=[{ \Pi}]^T[{ g}][{\Pi}]=-[\Pi][g][\Pi].
\end{equation}

Poisson manifolds endowed with symmetric contravariant 2-tensor fields are discussed in the literature, notably in the works of I.~Vaisman \cite{izu} and J.~P.~Morrison \cite{mo1,mo2}. The latter refers to these structures as metriplectic structures, which provide a geometric framework for dynamics that incorporates both conservative and dissipative features. 
\medskip

Next, we introduce the generalization of the DB vector field on a nonlinear Poisson manifold. In this setting, the DB vector field can be extended to accommodate the complexities inherent in nonlinear structures.  
\begin{defn} \label{dbv}
Let $(M,g, \Pi)$ be a pseudo-Riemannian manifold equipped with a Poisson structure. For a smooth function $G$ on $M$,  the associated vector field	\begin{equation}\label{defi}\partial_{\M}G:=-{i}_{\d G}{\M}\,,\end{equation}
is called the  \textbf{generalized double bracket (GDB) vector field}.
\end{defn}
\begin{rem}\label{r|}
	By Lemma \ref{bigdiagram}, we can observe that the GDB vector field $\partial_{\M}G$ is closely related to the Hamiltonian vector field $X_G$.
	For $G\in {\cal C}^{\infty}(M)$, we have
	\begin{equation}\label{calc}	\partial_{\M}G={(\sharp_\Pi \circ \flat_g) (X_G)} = 
		i_{\flat_g(X_G)}\Pi\,,
	\end{equation}
which implies that $\partial_{\M}G$ is tangent to all the symplectic leaves of $\Pi$.  Additionally, note that $\partial_{\M}G = -\sharp_\M (\d G)$ by \eqref{defi}.
\end{rem}

\medskip\noindent 
The GDB vector field $\partial_{\M}G$ serves as a natural generalization of the DB vector field defined on a semisimple Lie algebra $\g$. In this context, the Killing form provides an isomorphism $\flat_\ka\colon \g\to\g^*$, which facilitates the transportation of the Poisson structure from $\g^*$ to the linear Poisson structure on $\g$ as expressed in equation in \eqref{2}. 

\begin{lem}\label{xg}
On a semisimple Lie algebra $\g$, with the linear Poisson structure \eqref{2},  the Hamiltonian vector field associated with the function $G\in C^{\infty}(\g)$ is defined as
	\begin{equation}\label{xbarg}
		X_{ G}(L)=[L,\nabla G(L)],\text{ for all }L\in\g\,,
	\end{equation}
	where the gradient $\nabla G(L)$ is taken with respect to the Killing metric $\ka$ on $\g$.
\end{lem}

\begin{proof}
	The Hamiltonian vector fields associated  with Hamiltonian functions $G$ and ${\bar G}=G\circ \sharp_\ka$ on the dual space $\g^*$ are $\ka$-related. Therefore, it suffices to demonstrate that the Hamiltonian vector field 
$X_{\bar G} $ is
	\begin{equation*}
		X_{\bar G}(\xi)=\ka([L,\nabla G(L)]),\text{ for all }\xi=\flat_\ka(L)\in\g^*.
	\end{equation*}
	By the definition of the Lie-Poisson bracket on $\g^*$, we have
\[\Pi_\xi(L,L')=(\xi,[L,L'])\,,\]
	for all $L,L'\in\g\cong\g^{**}$.
	Thus, for $\xi=\flat_\ka(L)$, we obtain
	\begin{equation}\label{unu}
		\Pi_\xi(\nabla G(L),L')=(\xi,[\nabla G(L),L'])=\ka(L,[\nabla G(L),L'])=\ka([L,\nabla G(L)],L').
	\end{equation}
	On the other hand,  the Hamiltonian vector field associated with Hamiltonian function $\bar G$ satisfies
	\begin{equation}\label{doi}
		\Pi_\xi(\d\bar G(\xi),L')=(X_{\bar G}(\xi),L').
	\end{equation}	
	Since $\nabla G(L)=\d\bar G(\xi)$, where $\d\bar G(\xi)\in\g^{**}\cong\g$,
	the identity \eqref{xbarg} follows from equations \eqref{unu} and \eqref{doi}.	
\end{proof}

\begin{thm}\label{egalitate}
	Let $({\g},[\cdot,\cdot])$ be a semisimple Lie algebra endowed with the Poisson bracket defined by \eqref{2}. 
	Then the nonlinear DB vector field coincides with the classical DB vector field
	$$\partial_{\M}G(L)=[L,[L,\nabla G(L)]]\,.$$
\end{thm}

\begin{proof}
	As before, we work on $\g^*$, so let $\xi=\flat_\ka(L)$. The formula \eqref{calc} implies
	\begin{align*}
		(\pa_\M\bar G(\xi),L')&=\Pi_{\xi}(\flat_\ka(X_{\bar G})(\xi),L')=\Pi_{\xi}(X_{G}(L),L')=(\xi,[X_G(L),L'])\\
		&=\ka(L,[X_G(L),L'])=\ka([L,X_G(L)],L'),
	\end{align*}
	hence $\pa_\M\bar G(\flat_\ka(L))=\ka([L,X_G(L)])$. By using the fact that the nonlinear DB vector fields on $\g$ and $\g^*$ are $\ka$-related, 
	we get $\pa_\M G(L)=[L,X_G(L)]$. Now the Lemma \ref{xg} yields the conclusion. 
\end{proof}

\color{black}

\section{Gradient nature of the generalized double bracket vector field}
\label{Sec3}

\medskip
In this section, we demonstrate that the GDB vector field defined in Definition \ref{dbv}, when restricted to a leaf, behaves as a gradient vector field with respect to a metric that we refer to as the DB metric. This metric generalizes the normal metric defined on the regular adjoint orbits of a semisimple compact Lie algebra. It is important to note that when the ambient metric has an indefinite signature, caution is required when dealing with submanifolds, as they may not be pseudo-Riemannian with respect to the induced metric \cite{O'Neill}.

\medskip \noindent
Now, let $(M,g,\Pi)$ be a Poisson manifold equipped with a pseudo-Riemannian metric.	As a consequence of Lemma \ref{bigdiagram}, we have the inclusion $\Im \sharp_\M \subseteq \Im \sharp_\Pi$. Points where the two images do not coincide require more precision:
	
	\begin{defn} \label{Msing}
		A point $m\in M$ is called \textbf{$\M$-regular} if the image of the map $\Pi$ at $m$ coincides with the image of the map  $\sharp_\M$ at $m$, $\Im \sharp_\Pi \vert_m = \Im \sharp_\M\vert_m$.
        Otherwise, the point $m$ is referred to as \textbf{$\M$-singular}.
	\end{defn}
	
	\begin{prop}\label{singular}
		Let $S$ denote a symplectic leaf of $(M,g,\Pi)$, and let $\iota \colon  S \hookrightarrow M$ be the inclusion map. The $2$-tensor field induced by $g$ on $S$, $g_{\ind}^S := \iota^* g$, is degenerate at $s\in S$ if and only if $s$ is $\M$-singular.
	\end{prop}
	
	\begin{proof} Since $\srp \colon  T^*M \rightarrow TS$ is surjective, $g^S_{\ind}$ is non-degenerate
	if and only if $$g^S_{\ind}(\srp \alpha, \srp \beta ) = 0\,, \quad\forall\beta \in T^*M\,.$$ 
  It implies $\alpha \in \ker \srp$. The induced metric
	$g^S_{\ind}$ coincides with $g$ upon evaluation on vectors tangent to $S$. By  Definition \ref{Mdef}, we have
	$$ g^S_{\ind}(\srp \alpha, \srp \beta ) =\beta (\srm \alpha)\, ,  $$ for all 
	$\alpha, \beta \in T^*M$. The right-hand-side vanishes for all $\beta \in T^*M$ if and only if $\alpha \in \ker \srm$. 
	Consequently, $g^S_{\ind}$ is degenerate 
	if and only if $\Ker \srp \subsetneq \Ker \srm $. This is equivalent to the strict inclusion $\Im \srm \subsetneq \Im \srp $, hence the conclusion.
	\end{proof} 

{\begin{rem}\label{31}
\rm		For a Riemannian metric $g$, all points in $M$ are $\mathcal{M}$-regular. This implies that the $\mathcal{M}$-distribution $\Im\sharp_{\mathcal{M}}$ coincides with the characteristic distribution of the Poisson manifold, ensuring its integrability.
In contrast, , for a pseudo-Riemannian metric $g$ with signature, the integrability of the $\mathcal{M}$-distribution is not guaranteed. 
	\end{rem}
	
\begin{defn}
	We call a symplectic leaf $S\subset M$  a \textbf{good symplectic leaf} if the induced metric $g^S_{\ind}$ is non-degenerate.
	Equivalently, $S$ is a good symplectic leaf if all its points are $\M$-regular. 
	\end{defn}
    \begin{exmp}
For a Riemannian metric $g$ on $M$,  all symplectic leaves are good leaves. This is due to the fact that a Riemannian metric is positive-definite, ensuring that the induced metric on each symplectic leaf is non-degenerate and well-defined. Consequently, every symplectic leaf possesses the necessary geometric properties to be classified as a good leaf.
    \end{exmp}
\begin{exmp}
    Consider $M=\R^4 \ni (w,x,y,z)$ equipped with the pseudo-Riemannian metric
    $g = 2 \,\mathrm{d} w \mathrm{d} x + 2 \, \mathrm{d} y \mathrm{d} z$ 
    and the Poisson bivector field $\Pi := \partial_x \wedge \partial_y$. Each symplectic leaf $S$ is a plane 
    characterized by $w=w_0$, $z=z_0$ for some $(w_0,z_0)\in \R^2$. 
    On every $S$, since $\d w\vert_S = 0$ and $\d z\vert_S = 0$\,, the induced metric vanishes identically,
    $ g^S_{\ind} = 0\,.$ Correspondingly, as a straightforward matrix calculation demonstrates , see \eqref{prod}, 
    the metriplectic tensor field vanishes identically as well, 
     $ \M = 0$\,. Consequently,  $\Im \sharp_\M = \{ 0 \} \neq 
     \mathrm{Vect}(\partial_x,\partial_y) = \Im \sharp_\Pi$. Every point $m \in \R^4$ is $\M$-singular. This example contains no good symplectic leaves. 
  \end{exmp}   
  \begin{exmp}
  A classical example illustrating a combination of good symplectic leaves and symplectic leaves consisting exclusively of $\M$-singular points is the case of the non-compact semisimple Lie algebra $\mathfrak{sl}(2,\RR)$ with Lie-Poisson structure and Killing metric. The structure includes:
  \begin{itemize}
 \item The hyperboloids represent good symplectic leaves.  
 \item The two connected components of the cone, excluding the vertex, comprise only $\M$-singular points.
  \end{itemize}   
  \end{exmp}
 \noindent 
More sophisticated examples featuring symplectic leaves that contain both $\mathcal{M}$-regular and $\mathcal{M}$-singular points will be explored in future research.

\medskip
In what follows we will define the DB metric on a good symplectic leaf and we will prove that this metric generalizes the normal/standard metric on adjoint orbits of a semisimple compact Lie algebra.

\begin{defn}\label{double-bracket-metric}
	The \textbf{double bracket (DB) metric} on a good symplectic leaf $S$  is the pseudo-Riemannian metric
	defined as
	\begin{equation}		{\tau}_{\db}^{S}\left(X,Y\right):=({g}_{\ind}^{S})^{-1}\left({ i}_{X}\omega^S,{i}_{Y}\omega^S\right),\quad X,Y\in\X(S)\,, \label{DB}  
	\end{equation}		
	where $\omega^S$ is the induced symplectic form on the leaf $S$ and  $({g}_{\ind}^{S})^{-1}$ is the co-metric tensor field 
	associated to the pseudo-Riemannian metric ${g}_{\ind}^{S}$ induced on $S$ by the ambient metric.
\end{defn}

\begin{lem}\label{smalldiagram}
	The following identity 
	$\flat_{\ta^S_{\db}}=-\flat_{\om^S}\o\sharp_{g^S_\ind}\o\flat_{\om^S}$ holds on every good symplectic leaf.
\end{lem}

\begin{proof}
The proof proceeds through the following straightforward computation:
	\begin{align*}
		(\flat_{\ta^S_{\db}}X)(Y)&=\ta(X,Y)=({g}_{\ind}^{S})^{-1}\left({ i}_{X}\omega^S,{i}_{Y}\omega^S\right)
		=( \flat_{\om^S}Y)(\sharp_{\srind} \flat_{\om^S}X)\\
		&=\om^S(Y,\sharp_{\srind} \flat_{\om^S}X)
		=-\om^S(\sharp_{\srind} \flat_{\om^S}X,Y)
		=-\flat_{\om^S}\sharp_{\srind} \flat_{\om^S}(X)(Y)
	\end{align*}
	for all $X,Y\in\X(S)$. 
\end{proof}
\smallskip 
In the finite-dimensional case, for a point $x$ on the symplectic leaf $S$, the matrix  representation of the co-metric tensor field $({ g}_{\ind}^{S})^{-1}$ is given by
$$[{g}_{\ind}^{S}(x)]^{-1}=[{g}(x)|_{{T_xS \times T_xS}}]^{-1}\,.$$
Consequently, the matrix associated with the DB metric ${\tau}_{\db}^{S}$ can be expressed as
$$[{\tau}_{\db}^{S}(x)]=[\omega^S(x)]^{T}[{g}_{\ind}^{S}(x)]^{-1}[\omega^S(x)]=-[\omega^S(x)][{g}_{\ind}^{S}(x)]^{-1}[\omega^S(x)]\,.$$

\begin{lem}\label{leftdiagram}
	The following identity 
	$\srm=\io_*\o\sharp_{\ta_{\db}^S}\o \io^*$ holds for every good symplectic leaf $S$, where $\io\colon S\to M$ denotes the immersion of the leaf into the ambient manifold; see the diagram below.
\end{lem}

\begin{proof}
This can be shown through the following direct computation:
	\begin{align*}
		\srm&=-\srp\o\flat_g\o\srp=-\io_*\o\sharp_{\om^S}\o \io^*\o\flat_g\o \io_*\o\sharp_{\om^S}\o \io^*\\
		&=-\io_*\o\sharp_{\om^S}\o\flat_{g_{\ind}^S}\o\sharp_{\om^S}\o \io^*=\io_*\o\sharp_{\ta_{\db}^S}\o \io^*\,,
	\end{align*}
    where we used Lemma \ref{bigdiagram} in the first step and Lemma \ref{smalldiagram} in the fourth step.
\end{proof}	

 The commutative diagram below summarizes the lemmas \ref{bigdiagram}, \ref{smalldiagram}, and \ref{leftdiagram}.
  \begin{figure}[H]
    \centering
  %  \caption{Your caption here} % Add a caption for clarity
    \label{fig} % Place label after caption
    $$\xymatrix@C-10pt@R-5pt{
        &  T^*M|_S\ar[rrr]^{\sharp_{\Pi}} \ar[dr]^{\iota^*} \ar[ddl]^{\sharp_\M}
        & & & T M|_S\ar[ddl]\ar[ddl]^{\flat_{\g}} \\
        & & T^*S\ar[rrr]^{\sharp_{\omega^{S}}}\ar[ddl]^{\sharp_{\tau_{DB}^S}}
        & & &  T S \ar[ul]_{\iota_*}\ar[ddl]^{\flat_{{\g}_{ind}^S}}\\
        T M|_S
        & & &  T^* M|_S\ar[dr]^{\iota^*}\ar[lll]_{-\sharp_{\Pi}}& &  \\
        &  T S\ar[ul]_{\iota_*}& & & T^*S\ar[lll]_{-\sharp_{{\omega}^S}} &
    }$$
\end{figure}
\medskip

Now we will explore the relationship between the GDB vector field and the gradient in the context of Poisson manifold.
\begin{thm}\label{gradient th}
	Let $M$ be a smooth manifold equipped with a pseudo-Riemannian structure and  with a Poisson structure. 
	On a good symplectic leaf $S$, for $G\in C^\oo(M)$,
	the GDB vector field $\pa_\M G$, is the negative of the gradient vector field of $G|_{S}$ with respect to the DB metric:
	\begin{equation} \label{thmeq}
		(\pa_\M G)(x)=-\nabla_{\ta^S_{\db}} {(G|_{S})}(x),\quad x\in S\,. 
	\end{equation}
\end{thm}

\begin{proof}
	Using the Lemma \ref{leftdiagram}, we have,
	\begin{align*}
		(\partial_{\M}G)|_S&=(\srm(\d G))|_S=\io_*\sharp_{\ta_{\db}^S}(\io^*(\d G)|_S)\\
		&=\io_*\sharp_{\ta_{\db}^S}(\d(G|_S))=-\io_*\nabla_{\ta^S_{\db}} {(G|_{S})}\,,
	\end{align*}
	hence the equality in \eqref{thmeq}.
\end{proof}

\medskip
Next, we will demonstrate that in the case of a compact semisimple Lie algebra, the DB metric introduced above coincides with the normal metric (as defined in equation \eqref{nor}) up to a sign. 

\begin{thm}
	Let ${\g}$ be a semisimple compact Lie algebra and let $S$ be a regular adjoint orbit in $\g$. Then
	$${\tau}_{\db}^{S}=-{\nu}^{S}.$$
\end{thm}

\begin{proof}
We consider the linear Poisson structure \ref{2} on $\g$. Compactness implies that the Killing form has index zero; therefore, every symplectic leaf $S$ of $\g$ is a good symplectic leaf. It is sufficient to show that for all functions on the symplectic leaf $S\subset \g$ of the form $H(L)=\ka(L,N)$, with $N$ regular element in $S$, the gradient with respect to the normal metric  ${\nu}^{S}$ is the opposite of the gradient with respect to the DB metric $\boldsymbol{\tau}_{\db}^{S}$\,.

\noindent	By Theorem \ref{bloch-ratiu}, we have ${\nabla}_{\nu^S}H(L)=[L,[L,N]]$ for all $L\in S$.
	On the other hand, by Theorem \ref{egalitate}, we have
	$$\partial_{\M}H(L)=[L,[L,\nabla H(L)]]=[L,[L,N]]\,.$$
Combining these results with Theorem \ref{gradient th}, we obtain ${\nabla}_{\nu^S}H(L)=-{\nabla}_{\ta^S_{\db}}H(L)$, which implies that
	${\nu}^{S}=-{\tau}_{\db}^{S}$.
\end{proof}{\tiny }

\section{Asymptotic stabilization via generalized double bracket vector field}\label{sec:4}
In this section, we will demonstrate how the incorporation of a GDB vector field into Hamiltonian dynamics can renders stable equilibria into asymptotically stable equilibria, while preserving the structure of the symplectic leaves. This method of asymptotic stabilization has been previously utilized in various contexts, as discussed in \cite{BC, Bloch,Moris}. We will specifically apply this technique to the example of two harmonic oscillators in $(n:m)$ resonance, as presented in \cite{holm}. 

\smallskip
We begin with the framework of a Poisson manifold  endowed with a Riemannian metric $(M,\Pi,g)$. In this context, we construct the associated metriplectic tensor field $\M$. As noted in Remark \ref{31}, in this case, all symplectic leaves are classified as good leaves.
\smallskip

Given a Hamiltonian Poisson dynamics on $M$
\[
\dot x=X_H(x),
\]
assume that $x_0$ is a local minimum for the Hamiltonian function $H$ restricted to the symplectic leaf $S_0$ that contains $x_0$,
making it a stable equilibrium point for the dynamics. 
We consider the following dissipative system that preserves the symplectic leaves:
\[
\dot x=X_H(x)+f(x)\pa_{\M}H(x)
\]
with $f$ a smooth function strictly positive on $S_0$. 
By Theorem \ref{gradient th}, on the leaf $S_0$, the above dissipative system becomes
\begin{equation}\label{dis}
\dot x=X_{H_0}(x)-f_0(x)\nabla_{\ta^{S_0}_{\db}} {H_0}(x),
\end{equation}
where $H_0:=H|_{S_0}$ and $f_0:=f|_{S_0}$. Note that, the function $f_0(x)$ is a positive scalar that quantifies the strength of this dissipation. 
\smallskip

By the initial assumption, $x_0$ is a local minimum for $H_0$. 
 For the dissipative system described by equation \eqref{dis}, we compute that:
\[
\dot H_0=-f_0||\nabla_{\ta^{S_0}_{\db}} {H_0}||^2\,,
\]
which indicates that the rate of change of the Hamiltonian is proportional to the negative of the square of its gradient, scaled by a factor $f_0$. This reflects how energy is dissipated in directions where $H_0$ increases, confirming that dissipation acts to stabilize the system by reducing energy.

Thus, $H_0$ serves as a Lyapunov function. Therefore, by Lyapunov's theorem, the stable equilibrium point $x_0$ for the Hamiltonian dynamics becomes an asymptotically stable equilibrium for the dissipative system described by equation \eqref{dis}. 
\smallskip 

Next we apply the asymptotic stabilization procedure described above to the example of two harmonic oscillators in $(n:m)$ resonance.
%\begin{exmp}
Let the circle $S^1$ act on $\CC^2$ by $e^{i\th}\cdot(z_1,z_2)=(e^{in\th}z_1,e^{im\th}z_2)$, where $n$ and $m$ are integers. This defines a Hamiltonian action with a corresponding momentum map
\[
R:\CC^2 \to \R,\quad R(z_1,z_2)=\dfrac{n}{2}|z_1|^2+\dfrac{m}{2}|z_2|^2.
\]
The quotient map can be expressed using the $S^1$-invariant functions
\begin{equation*}
(X-iY)(z_1,z_2)=z_1^m\bar z_2^n\quad \text{ and }\quad 
Z(z_1,z_2)=\frac{n}{2}|z_1|^2-\frac{m}{2}|z_2|^2,
\end{equation*}
namely $p:=(X,Y,Z):(\CC\setminus\{0\})^2\to\RR^3\setminus Oz$.
The Poisson maps $R$ and $p$ form a dual pair of Poisson maps on $(\CC\setminus\{0\})^2$ \cite{HV}.

The reduced Poisson structure on $\RR^3\setminus Oz$ is given by
\[
\begin{array}{rcl}
\{y,z\}&=&2mnx,\\
\{z,x\}&=&2mny,\\
\{x,y\}&=&-mn(x^2+y^2)\left (\dfrac{m}{C(x,y,z)+z}-\dfrac{n}{C(x,y,z)-z}\right )\,,\\
\end{array}
\]
with the Casimir function $C$ implicitly defined by
\begin{equation}\label{cas}
x^2+y^2=\left(\frac{C(x,y,z)+z}{n}\right)^m\left(\frac{C(x,y,z)-z}{m}\right)^n\,.
\end{equation}
The symplectic leaves are the
Kummer surfaces: level sets of the Casimir function $C$.
 When $m=n=1$, we obtain the Lie-Poisson structure on $\mathfrak{su}(2)^*$\,. 
% In this case, the symplectic leaves are Kummer surfaces, which are spheres and represent the level sets of the Casimir function $C$.

The Riemannian metric on $\RR^3\setminus Oz$ for which the projection $p$ becomes a Riemannian submersion is the diagonal metric
$g={\rm Diag}(\be^{-1},\be^{-1},\ga^{-1})$,
with the notations
\begin{gather*}
\begin{array}{rcl}
\be&=&(x^2+y^2)^{\frac12}\left(m^2\left(\dfrac{C+z}{n}\right)^{-\frac12}+n^2\left(\dfrac{C-z}{m}\right)^{-\frac12}\right)\,,\\[10pt]
\ga&=&n(C(x,y,z)+z)+m(C(x,y,z)-z).
\end{array}
\end{gather*}
The additional notation 
\[
\al=(x^2+y^2)\left (\frac{m}{C(x,y,z)+z}-\frac{n}{C(x,y,z)-z}\right )
\]
allows us to express the matrix of the metriplectic tensor field as:
\[
[\M]=\begin{bmatrix} \dfrac{\al^2}{\be}+4m^2n^2\dfrac{y^2}{\ga}&-4m^2n^2\dfrac{xy}{\ga} & 2mn\dfrac{\al x}{\be} \\ 
-4m^2n^2\dfrac{xy}{\ga} & \dfrac{\al^2}{\be}+4m^2n^2\dfrac{x^2}{\ga}&2mn\dfrac{\al y}{\be} \\
2mn\dfrac{\al x}{\be} &2mn\dfrac{\al y}{\be} &4m^2n^2\dfrac{x^2+y^2}{\be}\end{bmatrix}\,.
\]
Given a smooth function $H$, we get the Hamiltonian vector field 
\[
X_H=(-\al H_y-2mnyH_z,\al H_x+2mnxH_z,2mn(yH_x-xH_y))\,,
\]
%\newpage
and so the GDB vector field can be expressed as:
\begin{align*}
\pa_{\M}H=-&((\tfrac{\al^2}{\be}+4m^2n^2\tfrac{y^2}{\ga})H_x-4m^2n^2\tfrac{xy}{\ga} H_y+ 2mn\tfrac{\al x}{\be}H_z,\\
&-4m^2n^2\tfrac{xy}{\ga} H_x+( \tfrac{\al^2}{\be}+4m^2n^2\tfrac{x^2}{\ga})H_y+2mn\tfrac{\al y}{\be} H_z,\\
&2mn\tfrac{\al x}{\be} H_x+2mn\tfrac{\al y}{\be} H_y+4m^2n^2\tfrac{x^2+y^2}{\be}H_z ).
\end{align*}
\noindent 
Adopting the Euclidean metric, in place of the Riemannian metric $g$ derived via the quotient map, leads to an expression for $\partial_{\mathcal{M}} H$ that omits the denominators $\beta$ and $\gamma$. 

In the special case $m=n=3$, the Casimir function $C$ in \eqref{cas} can be expressed as
\[
C(x,y,z)=(z^2+9(x^2+y^2)^{1/3})^{1/2}.
\] 
We consider the Hamiltonian function $H(x,y,z)=az^2-x$.
It is straightforward to verify that each point on the $Ox$ satisfies the condition of being a local minimum for $H$ when restricted to its symplectic leaf. Consequently, these points are stable equilibrium points for the Hamiltonian dynamics on their respective symplectic leaves. By adding a scaled GDB vector field to the initial Hamiltonian dynamics, as described in equation \eqref{dis}, these equilibria become asymptotically stable for this dissipative system.

\smallskip
This transformation is illustrated in Figure \ref{julia}: the stable equilibrium in (A) for $X_H$ is transformed into an asymptotically stable equilibrium in (B) for $X_H+c\pa_{\M}H$\,.

\begin{figure}[H]
\subfloat[a][]
	\centering
	\scalebox{0.45}{\includegraphics{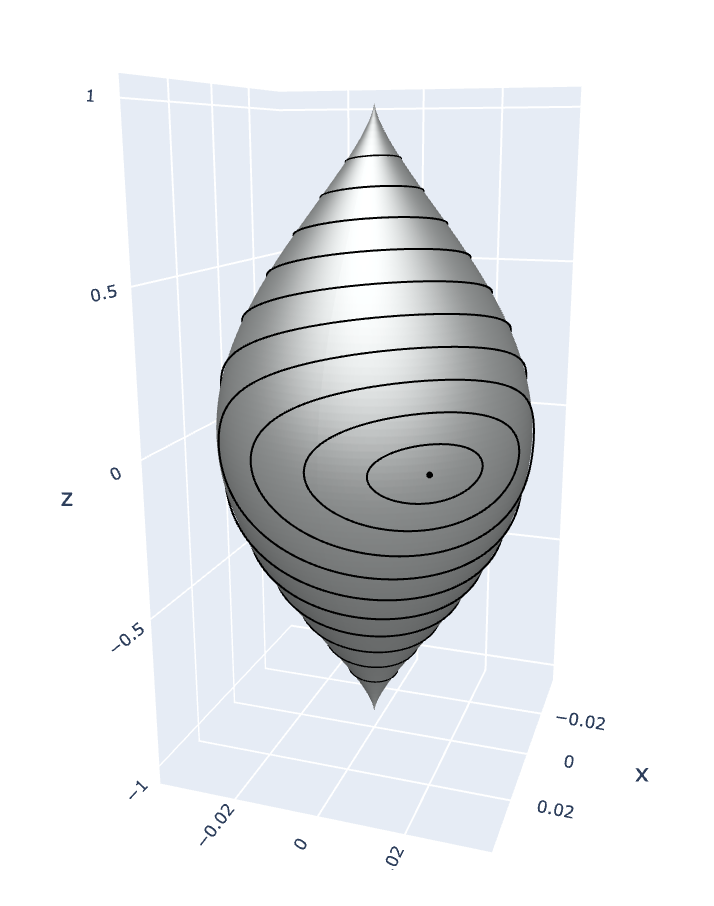}} 
	\subfloat[c][]
	{\includegraphics[scale=0.45]{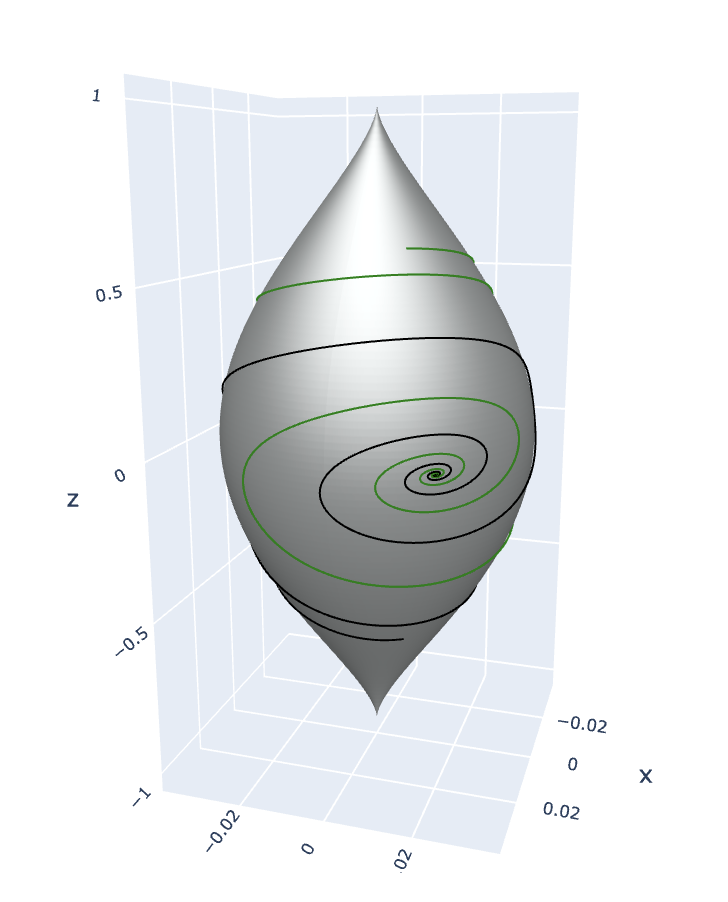}}
	\caption{Dynamics on a Kummer surface of type $(3:3)$ \\(A) Hamiltonian (B) Dissipative}
	\label{julia}
\end{figure}  
%\end{exmp}
\noindent {\bf Acknowledgements} 
	Z.~Ravanpak acknowledges the scholarship “Cercetare postdoctoral\u a avansat\u a” funded by the West University of Timi\c soara, Romania, the financial support from the Spanish Ministry of Science and Innovation under grants PID2022-137909NB-C22, and Erwin Schrödinger International Institute for Mathematics and Physics (ESI), University of Vienna,  where a part of this work has been done.
The authors are grateful to Thomas Strobl for valuable remarks and to Emilia Petri\c sor for the eloquent visualizations created using the Julia software libraries DifferentialEquations.jl and DynamicalSystems.jl.

\noindent {\bf Data availability statement} Data sharing not applicable to this article as no datasets were generated or
analyzed during the current study.

\noindent {\bf Conflict of interest} The author has no relevant financial or non-financial interests to disclose.

\noindent {\bf Ethical Approval} Not applicable.

\noindent {\bf Informed Consent} Not applicable.

\end{document}